\documentclass[a4paper,reqno,11pt,final]{article}
\usepackage{latexsym}
\usepackage{amsthm,amsmath,amssymb}

\usepackage{graphicx, color}

\usepackage{transparent}

\usepackage{tikz}
\usetikzlibrary{arrows}
\usepackage{rotating}
\usepackage{lscape}

\newtheorem{thm}{Theorem}

\newtheorem{defn}{Definition}

\newtheorem{prop}[thm]{Proposition}

\newtheorem{cor}[thm]{Corollary}

\newtheorem{lem}[thm]{Lemma}

\newcommand{\Q}{\mathbb{Q}}
\def\O{\mathcal O}

 \def\R{\mathbb R}  
\def\={\;=\;}  \def\+{\,+\,}  \def\Q{\Bbb Q}         
 
\def\A{\mathbb A}

\def\be{\begin{equation}}   \def\ee{\end{equation}}
\def\bes{\begin{equation*}}   \def\ees{\end{equation*}}

\newcommand{\ar}{\mathrm{ar}}

\newcommand{\fin}{\mathrm{fin}}

\newcommand{\Res}{\mathrm{Res}}

\newcommand{\ord}{\mathrm{ord}}

\newcommand{\law}{\longleftarrow}
\newcommand{\raw}{\longrightarrow}

\def\ov{\overline}

\newcommand{\dis}{\displaystyle}
\def\two{\twoheadrightarrow}

\begin{document}
\title{\bf {\Large{$H_\ar^1$ for  Arithmetic Surface is Finite}}}
\author{\bf{ K. Sugahara, L. Weng}}
\date{(February 18, 2016)}
\maketitle
\begin{abstract}
For an arithmetic surface X and a Weil divisor $D$, there are natural arithmetic 
cohomology groups $H_\ar^i(X, \O_X (D))$\ $(i=0,1,2)$. Using ind-pro topology on
adelic space $\A_{X, 012}^\ar$, we show that 
$H_\ar^0(X, \O_X (D))$ is discrete, $H_\ar^1(X, \O_X (D))$ is finite, and 
$H_\ar^2(X, \O_X (D))$ is compact. Moreover, we prove that all possible summations of canonical subspaces $\A_{X,i}^\ar(D),$\ $\A_{X, kl}^\ar(D)$ 
$(i,k,l=0,1,2)$  are closed in $\A_{X,012}^\ar$, and hence complete our proof of topological dualities of among 
$H^i_\ar$'s.
\end{abstract}

\section{Arithmetic Cohomology Groups}
In this section,  we review natural constructions and basic properties involving various canonical adelic spaces $\A_{X,*}^\ar(D)$'s and  arithmetic cohomology groups $H_\ar^i(X,\O_X(D))$\ $(i=0,1,2)$ 
associated to an arithmetic surface $X$ and a Weil divisor $D$. For details, please refer to [2].

Let $\pi: X\to \mathrm{Spec}\,(\O_F)$ be an arithmetic surface over  integer 
ring $\O_F$ of a number field $F$. We assume that $\pi$ is projective and 
flat, that $X$ is integral and regular. Denote by  $X_F$ its generic fiber, and  $k(X)$, resp. 
$k(X_F)$, the field of rational functions of $X$, resp. of $X_F$. We have $k(X)=k(X_F)$. 

Let $(X,C,x)$ be a flag of $X$, consisting of  an integral curve $C$ on $X$ and  a closed point 
$x$ of $C$). Denote by $k(X)_{C,x}$ its  local ring. It is known that $k(X)_{C,x}$ is a direct 
sum of two dimensional local fields.   Since $X$ is a Noetherian scheme, following Parshin-Beilinson ([4],\,[1]), we have a two dimensional 
adelic space $\A^{\fin}_{X,012}$, its level two subspaces $\A_{X,01}^\fin,\,\A_{X,02}^\fin,\,\A_{X,12}^\fin(D)$ and the associated level one subspaces $\A_{X,0}^\fin,\,\A_{X,1}^\fin(D),\,\A_{X,01}^\fin(D)$, for a Weil divisor $D$ on $X$. These spaces can be roughly described  as follows:
$$\begin{aligned}\A^{\fin}_X\!\!=\!\A_{X,012}^\fin\!\!:=&\!\A_{X;012}^\fin(\O_X)\!\!:=\!{\prod}_{(C,x):\,x\in C}'k(X)_{C,x}:={\prod}_C'\big({\prod}_{x:\, x\in C}'k(X)_{C,x}\big),\\
\A_{X,01}^\fin:=&\{(f_{C})_{C,x}\in \A_{X,012}\},\qquad  \A_{X,02}^\fin:=\{(f_{x})_{C,x}\in \A_{X,012}\},\\
\A_{X,12}^\fin(D):=&\big\{(f_{C,x})_{C,x}\in \A_{X,012}\,\big|\, \ord_C(f_{C,x})+\ord_C(D)\geq 0\ \forall C\subset X\big\},\\
\A_{X,0}^\fin\!\!:=\!\A_{X,01}^\fin\cap&\A_{X,02}^\fin,\,\A_{X,1}^\fin\!(D)\!\!:=\!\A_{X,01}^\fin\!\!\cap\!\A_{X,12}^\fin(D),\,\A_{X,2}^\fin(D)\!\!:=\!\A_{X,02}^\fin\!\!\cap\!\A_{X,12}^\fin(D).\end{aligned}$$  
It is well-known  that $\A_{X,012}$ admits a natural ind-pro structure
$$
\A_{X;012}^\fin(\O_X)=\lim_{\substack{\longrightarrow\\ D_1}}
\lim_{\substack{\longleftarrow\\ D_2:\,D_2\leq D_1}}\A_{X;12}(D_1)\big/\A_{X;12}(D_1).
$$

Furthermore, following [3], introduce the adelic space at infinity by
$$\A_X^\infty:=\A_{X_F}\widehat\otimes_\Q\R
:=\lim_{\substack{\longrightarrow\\ {D}_1}}\lim_{\substack{\longleftarrow\\ {D}_2:{D}_2\leq {D}_1}}
\Big(\big(\A_{X_F;1}({D}_1)\big/\A_{X_F;1}({D}_1)\big)\,\otimes _{\mathbb Q}\,\mathbb R\Big),
$$
and define
$$
\A_{X}^{\ar}:=\A_{X;012}^\ar:=\A_X^{\fin}\bigoplus\A_X^\infty.
$$
Then $\A_{X,012}^{\ar}$ admits three  level two subspaces 
$\A_{X,01}^{\ar},\, \A_{X,02}^{\ar},\, \A_{X,12}^\ar(D)$ ([2, \S2.3.1]), which can be described by 
$$
\begin{aligned} 
\A_{X,01}^{\ar}=&\big\{(f_{C,x})\times(f_P)\in\A_X^{\ar}\,\big|\,(f_{C,x})_{C,x}\!\!=\!\!(f_C)_{C,x}\!\in\!\A_{X,01}^{\fin},
 f_{E_P}\!\!=\!\!f_P\ \forall P\in X_F\big\},\\[0.2em]
\A_{X,02}^{\ar}=&\A_{X,02}^\fin\bigoplus k(X_F)\widehat\otimes_{\mathbb Q}\R,\ \ 
\A_{X,12}^{\ar}(D):=\A_{X,12}^\fin(D)\bigoplus \big(\A_{X_F,1}(D_F)\widehat\otimes_{\mathbb Q}\R\big).
\end{aligned}$$
Here, as usual, within the adelic space $\A_{X_F,01}$ for the curve $X_F/F$, we have 
$$\A_{X_F,1}(D_F):=\big\{(f_P)\in\A_{X_F,01} \,\big|\, \ord_P(f_P)+\ord_P(D_F)\geq 0\ \forall P\in X_F\big\}$$
and we, moreover, define
$$
\A_{X_F,1}(D_F)\,\widehat\otimes_{\mathbb Q}\,\R
:=\lim_{\substack{\law\\
 E_F: E_F\leq D_F}}\Big(\A_{X_F,1}(D_F)/\A_{X_F,1}(E_F)\otimes_{\mathbb Q}\R\Big),
$$
using  the natural diagonal imbeddings 
$k(X)\hookrightarrow \A_{X_F}\hookrightarrow\A_{X}^{\ar}$.
Similarly, we have three level one canonical subspaces
$$\A_{X,0}^\ar\!\!:=\!\A_{X,01}^\ar\cap\A_{X,02}^\ar,\ \A_{X,1}^\ar\!(D)\!\!:=\!\A_{X,01}^\ar\cap\A_{X,12}^\ar(D),\ 
\A_{X,2}^\ar(D)\!\!:=\!\A_{X,02}^\ar\cap\A_{X,12}^\ar(D).$$ 

Accordingly, there are natural ind-pro structures on $\A_{X,012}^\ar$ ([2,\S2.4.3]):
$$
\A_X^\ar=\lim_{\substack{\raw\\ D}}\lim_{\substack{\law\\ E:\, E\leq D}}\A_{X,12}^\ar(D)\big/\A_{X,12}^\ar(E),
$$ 
and on $\A_{X,01}^\ar$ and $\A_{X,02}^\ar$ ([2, Corollary 14]):
$$
\begin{aligned}
\A_{X,01}^\ar
=&\lim_{\substack{\raw\\ D}}\lim_{\substack{\law\\ E:\,E\leq D}}\A_{X,1}^\ar(D)\big/\A_{X,1}^\ar(E),\\
\A_{X,02}^\ar
=&\lim_{\substack{\raw\\ D}}\lim_{\substack{\law\\ E:\,E\leq D}}\A_{X,2}^\ar(D)\big/\A_{X,2}^\ar(E).
\end{aligned}
$$

Based on these genuine ind-pro structures on  arithmetic adelic spaces, we may introduce natural ind-pro topologies on these spaces, as what we do for locally compact spaces. With
topologies hence defined, following [2, \S\S3.1.2-3], particularly, [2, Proposition 26], we know that natural inclusions of the above canonical level two and hence level one subspaces to the total space are all continuous with closed images. Consequently, these subspaces are also Hausdorff, since  $\A_{X,012}^\ar$ is Hausdorff by [2, Theorem II]. For later references, we summarize this as 

\begin{prop}([2, Proposition 26])
Let $D$ be a Weil divisor on  $X$. Its canonical adelic spaces
$\A_{01}^\ar,\, \A_{02}^\ar,\, \A_{12}^\ar(D),\, \A_{0}^\ar,\,\A_{1}^\ar(D),\,\A_{2}^\ar(D)$ are closed and Hausdorff.  
\end{prop}

We introduce arithmetic cohomology 
groups as follows:
\begin{defn} ([2,\S2.4.1])  Let $D$ be a Weil divisor on $X$. Then we define
arithmetic cohomology groups $H_\ar^i(X,D)$ of  $D$ on $X$\ $(i=0,1,2)$ by
$$\begin{aligned}H^0_{\ar}(X,D):=&\,\A_{X, 01}^{\ar}\cap \A_{X, 02}^{\ar}\cap \A_{X, 12}^{\ar}(D);\\
H^1_{\ar}(X,D) :=&\\
=\,\big(\big(\A_{X, 01}^{\ar}+&\A_{X, 02}^{\ar}\big)\cap \A_{X, 12}^{\ar}(D)\big)\big/
\big(\A_{X, 01}^{\ar}\cap \A_{X, 12}^{\ar}(D)+\A_{X, 02}^{\ar}\cap \A_{X, 12}^{\ar}(D)\big);\\
H^2_{\ar}(X,D):=&\,\A_{X, 012}^{\ar}
\big/\big(\A_{X, 01}^{\ar}+ \A_{X, 02}^{\ar}+\A_{X, 12}^{\ar}(D)\big).\end{aligned}$$
\end{defn}

In the sequel, to simplify notations, when write $\A_{X,*}^\ar$, we will omit $X$. 
For example, we write $\A_{X,012}^\ar$ simply as $\A_{012}^\ar$.

\section{Adelic Sub-Quotient Spaces}

\begin{lem} Let $E$ be a Weil divisor on $X$. Then we have
\begin{itemize}
\item [(1)] $\A_{01}^\ar\big/\A_1(E)^\ar$ is discrete;
\item [(2)] $\A_{12}^\ar(E)\big/\A_1(E)^\ar$ is compact.
\end{itemize}
\end{lem}
\begin{proof}
(1) Using inductive limit, we have
$$\A_{01}^\ar\big/\A_1(E)^\ar=\lim_{\substack{\longrightarrow \\
D:\,D\geq E}}\A_{1}^\ar(D)\big/\A_1(E)^\ar$$ where $D$ runs over all Weil divisors on $X$ such 
that $D-E$ is effective. Note that if $D=E+C$ for an integral curve on $X$, either horizontal or 
vertical, we have $\A_{1}^\ar(D)\big/\A_1(E)^\ar\simeq k(C)=\A_{C,0}^\ar$, which is discrete (say, 
in $\A_{C,01}^\ar$). Hence, for any $D\geq E$, the quotient space 
$\A_{1}^\ar(D)\big/\A_1(E)^\ar$ is discrete as there exist finitely many $C_i$'s such that 
$D-E=\sum_i C_i$. To complete proof, recall that topologies on our adelic spaces are induced 
from the ind-pro one, that is, the strongest topology making all inductive limits continuous. Consequently, as an inductive limit of discrete spaces,  
$\A_{01}^\ar\big/\A_1(E)^\ar$ is discrete.

\noindent
(2) From exact sequence 
$$
\A_{1}^\ar(E)\big/\A_1^\ar(E')\to \A_{12}^\ar(E)\big/\A_{12}^\ar(E')
\twoheadrightarrow 
\frac{\A_{12}^\ar(E)\big/\A_{12}^\ar(E')}{\A_1^\ar(E)\big/\A_1^\ar(E')},
\eqno(*)
$$
by taking projective limit on $E'$, we obtain an exact sequence
$$
\A_1^\ar(E)\to \A_{12}^\ar(E)\twoheadrightarrow  \A_{12}^\ar(E)\big/ \A_1^\ar(E).
$$
(One may also see this using first 
$\frac{\A_{12}^\ar(E)\big/\A_{12}^\ar(E')}{\A_1^\ar(E)\big/\A_1^\ar(E')}\simeq \frac{\A_{12}^\ar(E)}{\A_1^\ar(E)+\A_{12}^\ar(E')}$ and then
$\A_{12}^\ar(E)\big/ \A_1^\ar(E)={\dis{\lim_{\substack{\longleftarrow\\ E':E'\leq E}}}}\frac{\A_{12}^\ar(E)}{\A_1^\ar(E)+\A_{12}^\ar(E')}$.)
Hence to prove $\A_{12}^\ar(E)\big/ \A_1^\ar(E)$ is compact, being the projective limit, it suffices 
to show that the quotient spaces 
$\frac{\A_{12}^\ar(E)\big/\A_{12}^\ar(E')}{\A_1^\ar(E)\big/\A_1^\ar(E')}$ are compact.
On the other hand, as above, for any integral curve $C$ on $X$,  there is a canonical exact 
sequence
$$
0\to \A_{C,0}^\ar\to \A_{C,01}^\ar\to \A_{C,01}^\ar\big/\A_{C,0}^\ar\to 0
\eqno(**)
$$
with $\A_{C,01}^\ar\big/\A_{C,0}^\ar$ compact. That is to say, 
$\frac{\A_{12}^\ar(E)\big/\A_{12}^\ar(E-C)}{\A_1^\ar(E)\big/\A_1^\ar(E-C)}$ is compact, since 
exact sequences $(*)$ and $(**)$ are equivalent when $E=E'+C$. Consequently,  
$\frac{\A_{12}^\ar(E)\big/\A_{12}^\ar(E')}{\A_1^\ar(E)\big/\A_1^\ar(E')}$ are compact for any Weil 
divisors $E\geq E'$, by first writing $E-E'=\sum_iC_i\geq 0$, then arguing in the same way as in 
proof of (1) above.
\end{proof}

\begin{prop} Let $E$ be  a Weil divisor on $X$. Then as subspaces of $\A_{012}^\ar$,
we have
\begin{itemize}
\item [(1)]  $\A_{01}^\ar(E)+\A_{12}^\ar(E)$  is closed;
\item [(2)] $\A_{0}^\ar+\A_{1}^\ar(E)$  is closed;
\item [(3)]  $\A_{1}^\ar(E)+ \A_{2}^\ar(E)$ is closed.
\end{itemize} 
\end{prop}
\begin{proof} (1) Recall that $\A_1(E)^\ar:=\A_{01}^\ar\cap\A_{12}^\ar(E)$, and by Lemma 2(2), 
the quotient $\A_{12}^\ar(E)\big/\A_1^\ar(E)$ is compact. Hence, its topological dual   is  discrete 
and given by $\frac{\ \ \ov{\A_{01}^\ar+\A_{12}^\ar((\omega)-E)}\ \ }{\A_{12}^\ar((\omega)-E)}$. 
As a subspace, $\frac{\ {\A_{01}^\ar+\A_{12}^\ar((\omega)-E)}\ }{\A_{12}^\ar((\omega)-E)}$
is then discrete. Since $\A_{12}^\ar((\omega)-E)$ is closed, and all quotient spaces involved 
here are Hausdorff, $\A_{01}^\ar+\A_{12}^\ar((\omega)-E)$ is closed.  So is 
$\A_{01}^\ar+\A_{12}^\ar(E)$, since $E$ is arbitrary above.

\noindent
(2) Similarly, by Lemma 2(1),  the quotient $\A_{01}^\ar\big/\A_1^\ar(E)$ is discrete. Hence 
$\A_{01}^\ar\big/\big(\A_0^\ar+\A_1^\ar(E)\big)$ is discrete. Clearly, it is also Hausdorff. Thus 
with the fact that $\A_{01}^\ar$ is closed, we have $\A_{0}^\ar+\A_{1}^\ar(E)$  is closed.

\noindent
(3) By definition,
$H^0_\ar(X,\O_X(E))=\A_{01}^\ar\cap \A_{02}^\ar\cap\A_{12}^\ar(E)
=\big(\A_{01}^\ar\cap \A_{12}^\ar(E)\big)\cap\big(\A_{02}^\ar\cap\A_{12}^\ar(E)\big)
=\A_{1}^\ar(E)\cap \A_{2}^\ar(E)$.
Hence there exists a natural surjection 
$\A_{2}^\ar(E)\big/H^0_\ar(X,\O_X(E))
\twoheadrightarrow \big(\A_{1}^\ar(E)+ \A_{2}^\ar(E)\big)\big/\A_{1}^\ar(E)$.  By Lemma 4 below, $\A_{2}^\ar(E)\big/H^0_\ar(X,\O_X(E))$ is compact.
So is $\big(\A_{1}^\ar(E)+ \A_{2}^\ar(E)\big)\big/\A_{1}^\ar(E)$. It is also Hausdorff, since $\A_{1}^\ar(E)$ is closed and $\A_{012}^\ar$ is Hausdorff. Therefore, $\A_{1}^\ar(E)+ \A_{2}^\ar(E)$ is closed;
\end{proof}

\begin{lem} For a Weil divisor $E$ on $X$, 
$\A_{2}^\ar(E)\big/H^0_\ar(X,\O_X(E))$ is compact.
\end{lem}

\begin{proof}
This can be deduced from the natural exact sequence
{\footnotesize{$$
0\to \frac{H_\ar^0(X,\O_X(E))+\A_2^\fin(E)}{H_\ar^0(X,\O_X(E))}\longrightarrow \frac{\A_2^\ar(E)}{H_\ar^0(X,\O_X(E))}\longrightarrow\frac{\A_2^\ar(E)}{H_\ar^0(X,\O_X(E))+\A_2^\fin(E)}\to 0.
$$}}
\!\!Indeed, $\frac{H_\ar^0(X,\O_X(E))+\A_2^\fin(E)}{H_\ar^0(X,\O_X(E))}$ is nothing but $\A_2^\fin(E)$, since, by definition, $H_\ar^0(X,\O_X(E))\cap\A_2^\fin(E)=\{0\}$.
Moreover, $\A_2^\fin(E)={\dis{\lim_{\substack{\longleftarrow\\ E}}}}\frac{\A_2^\fin(E)}{\A_2^\fin(E')}$. Thus, by the fact that, for an integral curve $C$ on $X$, $\frac{\A_2^\fin(E)}{\A_2^\fin(E-C)}\simeq\A_{C,1}^\fin(E|_C)$ is compact, we conclude that $\frac{\A_2^\fin(E)}{\A_2^\fin(E')}$ and hence $\frac{H_\ar^0(X,\O_X(E))+\A_2^\fin(E)}{H_\ar^0(X,\O_X(E))}$ are compact as well.
On the other hand, by definition, $\A_{X,2}^\ar(E)=\A_{X,2}^\fin(E)+H^0(X,\O_X(E))\otimes_\mathbb Z\mathbb R$ and
$H_\ar^0(X,\O_X(E))+\A_2^\fin(E)=H^0(X,\O_{X}(E))+\A_2^\fin(E)$. Hence, we have
 $\frac{\A_2^\ar(E)}{H_\ar^0(X,\O_X(E))+\A_2^\fin(E)}=\frac{\A_{X,2}^\fin(E)+H^0(X,\O_X(E))\otimes_\mathbb Z\mathbb R}{\A_2^\fin(E)+H^0(X,\O_{X}(E))}$ is naturally isomorphic to $\frac{H^0(X,\O_X(E))\otimes_\mathbb Z\mathbb R}{H^0(X,\O_X(E))}$, a compact torus.
\end{proof}

\begin{thm} 
For a Weil divisor $D$ on $X$,  $H^1_\ar(X,\O_X(D))$ is  finite.
\end{thm}

\begin{proof}
By proof of Proposition 3(2),  
$\frac{\A_{01}^\ar\cap\big(\A_{02}^\ar+\A_{12}^\ar(D)\big)}{\A_0^\ar+\A_1^\ar(D)}=H^1_\ar(X,\O_X(D))$, is  discrete.
On the other hand, $\frac{\A_{12}^\ar(D)}{\A_1(D)^\ar+\A_2(D)^\ar}$ is compact, since, by Lemma 2(2), $\A_{12}^\ar(D)\big/\A_1(D)^\ar$ is compact.  By Proposition 3(3), resp. Proposition 1, $\A_1(D)^\ar+\A_2(D)^\ar$, resp. $\A_{12}^\ar(D)$, is closed. So $\frac{\A_{12}^\ar(D)}{\A_1(D)^\ar+\A_2(D)^\ar}$ is  Hausdorff, since $\A_{012}^\ar$ is Hausdorff.
By Proposition 3(1), $\A_{01}^\ar+\A_{12}^\ar(D)$  is closed. Hence
$\frac{\A_{12}^\ar(D)\cap\big(\A_{01}^\ar+\A_{12}^\ar(D)\big)}{\A_1(D)^\ar+\A_2(D)^\ar}$ is closed and hence compact. But this later sub-quotient is another expression of $H^1_\ar(X,\O_X(D))$. Hence, $H^1_\ar(X,\O_X(D))$  is both compact and discrete. Therefore,  it  is  finite.
\end{proof}

\begin{cor} For a Weil divisor $E$ on $X$,
 $\A_{0}^\ar+\A_{2}^\ar(E)$  is closed.
\end{cor}
\begin{proof} Note that $H^1(X,\O_X(E))$ can also be written as
$\frac{\A_{02}^\ar\cap\big(\A_{01}^\ar+\A_{12}^\ar(E)\big)}{\A_0^\ar+\A_2(E)^\ar}$. 
By Proposition 3(1),  $\A_{01}^\ar+\A_{12}^\ar(E)$ is closed. Hence, the numerator is also 
closed. But $H^1_\ar$ is finite and Hausdorff, hence the denominator, namely, $\A_{0}^\ar+\A_{2}^\ar(E)$  is 
closed.
\end{proof}

\begin{lem}
Let $E$ be a Weil divisor on $X$. Then we have
\begin{itemize}
\item [(1)]  $\A_{0}^\ar+ \A_{12}^\ar(E)$ is closed;
\item [(2)] $\A_{1}^\ar(E)+\A_{02}^\ar$  is closed;
\item [(3)] $\A_{2}^\ar(E)+\A_{01}^\ar$  is closed.
\end{itemize}
\end{lem}
\begin{proof} (1) By Lemma 2(2), $\A_{12}^\ar(E)\big/\A_{1}^\ar(E)$  is compact.
With the natural surjection
$\A_{12}^\ar(E)\big/\A_{1}^\ar(E)\two\frac{ \A_{12}^\ar(E)+\A_{0}^\ar}{\A_{1}^\ar(E)+\A_{0}^\ar}$,
the later quotient space is also compact. Clearly, it is also Hausdorff. By Proposition 3(2), 
the denominator $\A_{1}^\ar(E)+\A_{0}^\ar$ is closed. Hence the numerator 
$\A_{0}^\ar+ \A_{12}^\ar(E)$ is closed as well.

\noindent
(2) Since 
$\A_1^\ar(D)+\A_{02}^\ar={\dis{\lim_{\substack{\longrightarrow\\ D'}}\lim_{\substack{\longleftarrow\\
{E: E\leq D'}}}}}
\frac{\big(\A_1^\ar(D)+\A_{02}^\ar\big)\cap\A_{12}^\ar(D')}
{\big(\A_1^\ar(D)+\A_{02}^\ar\big)\cap\A_{12}^\ar(E)}$, 
it suffices to show that 
$\frac{\big(\A_1^\ar(D)+\A_{02}^\ar\big)\cap\A_{12}^\ar(D')}
{\big(\A_1^\ar(D)+\A_{02}^\ar\big)\cap\A_{12}^\ar(E)}$, or more strongly, 
$\big(\A_1^\ar(D)+\A_{02}^\ar\big)\cap\A_{12}^\ar(E)$ is closed.

We first assume that $E\geq D$. In this case, 
$\big(\A_1^\ar(D)+\A_{02}^\ar\big)\cap\A_{12}^\ar(E)=\A_1^\ar(D)+\A_{2}^\ar(E)$. By Lemma 4,  $\A_2^\ar(E)\big/H^0_\ar(X,\O_X(E))$ is compact. This implies that 
$ \frac{\A_1^\ar(D)+\A_{2}^\ar(E)}{\A_1^\ar(D)+H^0_\ar(X,\O_X(E))}$ is compact, since there is a natural surjection 
$\A_2^\ar(E)\big/H^0_\ar(X,\O_X(E))\two  
\frac{\A_1^\ar(D)+\A_{2}^\ar(E)}{\A_1^\ar(D)+H^0_\ar(X,\O_X(E))}$. 
Thus, it  suffices to show that 
$\A_1^\ar(D)+H^0_\ar(X,\O_X(E))$ is closed. Now, by Lemma 2(2), the quotient space 
$\A_{12}^\ar(D)\big/ \A_{1}^\ar(D)$ is compact. Hence, $\frac{\A_{01}^\ar+\A_{12}^\ar(D)}{\big(\A_{01}^\ar+\A_{12}^\ar(D)\big)\cap\big(\A_{01}^\ar+\A_{02}^\ar+\A_{12}^\ar(E)\big)}$ is compact, from surjection
$\A_{12}^\ar(D)\big/ \A_{1}^\ar(D)\two \frac{\A_{01}^\ar+\A_{12}^\ar(D)}{\big(\A_{01}^\ar+\A_{12}^\ar(D)\big)\cap\big(\A_{01}^\ar+\A_{02}^\ar+\A_{12}^\ar(E)\big)}$.
Furthermore, with the closeness of the spaces  $\A_{01}^\ar+\A_{12}^\ar(D)$, resp. 
$\A_{01}^\ar+\A_{02}^\ar+\A_{12}^\ar(E)$, as proved in Proposition 3(1), resp. Proposition 8(3) below, whose proof is independent of this argument, 
we conclude that this latest quotient space 
$\frac{\A_{01}^\ar+\A_{12}^\ar(D)}
{\big(\A_{01}^\ar+\A_{12}^\ar(D)\big)\cap\big(\A_{01}^\ar+\A_{02}^\ar+\A_{12}^\ar(E)\big)}$ 
is nothing but the topological dual of $\big(\A_1^\ar(D)+H^0(X,\O_X(E))\big)\big/\A_1^\ar(D)$. 
So $\big(\A_1^\ar(D)+H^0(X,\O_X(E))\big)\big/\A_1^\ar(D)$ is discrete. Therefore, 
$\A_1^\ar(D)+H^0(X,\O_X(E))$ is closed, as required.

To complete the proof, we next treat general $E$. Fix a Weil divisor $D'$ on $X$ such that 
$D'\geq D,\, D'\geq E$. By the special case just proved, we have 
$\big(\A_1^\ar(D)+\A_{02}^\ar\big)\cap\A_{12}^\ar(D')$ is closed.  Since
$\big(\A_1^\ar(D)+\A_{02}^\ar\big)\cap\A_{12}^\ar(E)
=\big(\big(\A_1^\ar(D)+\A_{02}^\ar\big)\cap\A_{12}^\ar(D')\big)\cap \A_{12}^\ar(E)$, 
$\big(\A_1^\ar(D)+\A_{02}^\ar\big)\cap\A_{12}^\ar(E)$ is closed. 

\noindent
(3) By Lemma 4, $\A_{2}^\ar(E)\big/H^0_\ar(X,\O_X(E))$ is compact.
With the natural surjection 
$\A_{2}^\ar(E)\big/H^0_\ar(X,\O_X(E))\two \big(\A_{2}^\ar(E)+\A_{01}^\ar\big)\big/\A_{01}^\ar$, 
this latest space is also compact. Clearly, it is Hausdorff as well. Consequently, with denominator 
being closed, the numerator $\A_{2}^\ar(E)+\A_{01}^\ar$ is closed.
\end{proof}

\begin{prop} Let $E$ be a Weil divisor on $X$. Then we have
\begin{itemize}
\item [(1)]  $\A_{01}^\ar+\A_{02}^\ar$ is closed;
\item [(2)] $\A_{02}^\ar+\A_{12}^\ar(E)$ is closed;
\item [(3)] $\A_{01}^\ar+\A_{02}^\ar+\A_{12}^\ar(E)$ is closed;
\item [(4)] $\A_{0}^\ar+\A_{1}^\ar(E)+\A_{2}^\ar(E)$ is closed.
\end{itemize}
\end{prop}
\begin{proof}
(1) By Theorem 5, 
$H^1(X,\O_X(E))=\frac{\A_{12}^\ar(E)\cap\big(\A_{01}^\ar+\A_{02}^\ar\big)}
{\A_1(E)^\ar+\A_2(E)^\ar}$, is finite. Hence, the numerator 
$\A_{12}^\ar(E)\cap\big(\A_{01}^\ar+\A_{12}^\ar(E)\big)$ is closed since the denominator
$\A_1(E)^\ar+\A_2(E)^\ar$ is closed by Proposition 3(3). Consequently, for any two Weil divisors 
$D, \,E$ with $D\geq E$, viewing as subspaces of $\A_{12}^\ar(D)\big/\A_{12}^\ar(E)$, the space 
$\frac{\A_{12}^\ar(D)\cap\big(\A_{01}^\ar+\A_{02}^\ar\big)}
{\A_{12}^\ar(E)\cap\big(\A_{01}^\ar+\A_{02}^\ar\big)}$ is closed.
On the other hand, by definition,
$$
\A_{01}^\ar+\A_{02}^\ar
=\lim_{\longrightarrow D}\lim_{\longleftarrow E: E\leq D}
\frac{\A_{12}^\ar(D)\cap\big(\A_{01}^\ar+\A_{02}^\ar\big)}
{\A_{12}^\ar(E)\cap\big(\A_{01}^\ar+\A_{02}^\ar\big)}.
$$ 
Therefore, by definition of  ind-pro topology, the subspace $\A_{01}^\ar+\A_{02}^\ar$ is 
closed.

\noindent
(2) By Lemma 2(2), $\A_{12}^\ar(E)\big/\A_{1}^\ar(E)$ is compact. With natural surjection
$\A_{12}^\ar(E)\big/\A_{1}^\ar(E)
\two\frac{ \A_{02}^\ar+\A_{12}^\ar(E)}{ \A_{02}^\ar+\A_{1}^\ar(E)}$, 
the latest space is compact as well. Thus as above, since, by Lemma 7(2), the denominator is 
closed, the numerator $\A_{02}^\ar+\A_{12}^\ar(E)$ is closed.

\noindent
(3) By Lemma 2(2), $\A_{12}^\ar(E)\big/\A_{1}^\ar(E)$ is compact. With natural surjection
$\A_{12}^\ar(E)\big/\A_{1}^\ar(E)\two\frac{ \A_{01}^\ar+\A_{02}^\ar+\A_{12}^\ar(E)}
{ \A_{01}^\ar+\A_{02}^\ar+\A_{1}^\ar(E)}$, the latest space is compact as well. But 
$\A_{01}^\ar\supset \A_{1}^\ar(E)$, hence the denominator is simply $\A_{01}^\ar+\A_{02}^\ar$. 
By (1), it is closed. Consequently, as above,  the numerator 
$\A_{01}^\ar+\A_{02}^\ar+\A_{12}^\ar(E)$ is closed.

\noindent
(4) By  Lemma 4, $\A_{2}^\ar(E)\big/H^0_\ar(X,\O_X(E))$ is compact.
By definition, we have $H^0_\ar(X,\O_X(E))=\A_{0}^\ar\cap \A_{1}^\ar(E)\cap\A_{2}^\ar(E)$. 
Hence, from composition of natural surjections 
$\frac{\A_{2}^\ar(E)}{H^0_\ar(X,\O_X(E))}\two 
\frac{\A_{2}^\ar(E)}{\big(\A_{0}^\ar+ \A_{1}^\ar(E)\big)\cap\A_{2}^\ar(E))}\two 
\frac{\A_{0}^\ar+\A_{1}^\ar+\A_{2}^\ar(E)}{\A_{0}^\ar+\A_{1}^\ar(E)}$, 
we conclude that $ \frac{\A_{0}^\ar+\A_{1}^\ar(E)+\A_{2}^\ar(E)}{\A_{0}^\ar+\A_{1}^\ar(E)}$ is compact. 
By Proposition 3(2), $\A_{0}^\ar+\A_{1}^\ar(E)$ is closed. Consequently, $\A_{0}^\ar+\A_{1}^\ar(E)+\A_{2}^\ar(E)$ is closed as well.
\end{proof}

\begin{thm} Let $D$ be a Weil divisor on $X$. Then we have
\begin{itemize}
\item [(1)] $\A_0^\ar$ is discrete. In particular, $H^0_\ar(X,\O_X(D))$ is discrete;
\item [(2)]  $H^2(X,\O_X(D))$ is compact .
\end{itemize}
\end{thm}

\begin{proof} (1) By Lemma 2(2), $\A_{12}^\ar(E)\big/\A_{1}^\ar(E)$ is compact. Taking its 
topological dual, we have $\big(\A_{01}^\ar+ \A_{12}^\ar(E)\big)\big/\A_{12}^\ar(E)$ is discrete,  
since $\A_{1}^\ar(E)=\A_{01}^\ar\cap \A_{12}^\ar(E)$ and, by Proposition 3(1), 
$\A_{01}^\ar+ \A_{12}^\ar(E)$ is closed. Consequently,  
$\big(\A_{0}^\ar+ \A_{12}^\ar(E)\big)\big/\A_{12}^\ar(E)$ is discrete, since we have a natural 
injection 
$\big(\A_{0}^\ar+ \A_{12}^\ar(E)\big)\big/\A_{12}^\ar(E)
\hookrightarrow \big(\A_{01}^\ar+ \A_{12}^\ar(E)\big)\big/\A_{12}^\ar(E)$. In particular, 
$\A_{0}^\ar\big/H^0(X,\O_X(E))$ is discrete since, by the second isomorphism theorem, there is 
a natural isomorphism 
$\big(\A_{0}^\ar+ \A_{12}^\ar(E)\big)\big/\A_{12}^\ar(E)\simeq \A_{0}^\ar\big/H^0(X,\O_X(E))$. 
By definition, $H^0(X,\O_X(E))\subset H^0(X_F,\O_{X_F}(E_F))$, where $\O_{X_F}(E_F)$ 
denotes the restriction of $\O_X(E)$ to the generic fiber $X_F$ of arithmetic surface
$X\to \mathrm{Spec}\, \O_F$. Hence, we may choose $E$ with negative enough 
$\O_{X_F}(E_F)$ so as to get a vanishing $H^0(X,\O_X(E))$. This then implies that 
$\A_0^\ar$ and hence also $H^0(X,\O_X(D))$ are discrete. 
 
\noindent
(2) Taking  topological dual, we see that $\A_{012}^\ar\big/\big(\A_{01}^\ar+\A_{02}^\ar\big)$ is 
compact, since, by Proposition 8(1),  $\A_{01}^\ar+\A_{02}^\ar$ is closed.  This then implies that 
$H^2(X,\O_X(D))$ is compact, since we have a natural surjection 
$\A_{012}^\ar\big/\big(\A_{01}^\ar+\A_{02}^\ar\big)\two 
\A_{012}^\ar\big/\big(\A_{01}^\ar+\A_{02}^\ar+\A_{12}^\ar(D)\big)=H^2(X,\O_X(D))$.
\end{proof}

\section{Topological Duality of $H_\ar^i(X,\O_X(D))$'s}

In this section, we use adelic theory for arithmetic surfaces developed in [2] and
closeness of various natural subspaces of $\A_X^\ar=\A_{012}^\ar$ proved in the previous section to establish the following fundamental 

\begin{thm}
Let $X$ be an arithmetic surface with a canonical divisor $K_X$ and $D$ be a Weil divisor on $X$. Then we have the following
topological dualities
$$
\widehat {H_\ar^i(X,\O_X(D))}\simeq H_\ar^{2-i}(X,\O_X(K_X-D))\qquad i=0,1,2.
$$
\end{thm}

\noindent
{\it Remark.} A proof of this theorem is stated in [2, \S3.2.4]. 
However, a key condition on closeness of two sum spaces used is omitted in a well-known 
lemma cited there, e.g. Lemma 12 below. This was pointed out by Osipov at the end of January 
2016. Now with such closeness confirmed in \S 2, we can complete our proof.

Recall that
there is  a natural global residue pairing on 
$\A_{X,012}^\ar$, introduced using local residue theory. Indeed, by [2,\S2.2], for a fixed 
non-zero rational differential $\omega$ on $X$, we can define a global pairing on 
$\A_{X,012}^\ar$ with respect to $\omega$ by
$$
\begin{matrix}
\langle\cdot,\cdot\rangle_\omega:\quad \A_{X,012}^\ar\times \A_{X,012}^\ar&\longrightarrow& \mathbb S^1\\
\big((f_{C,x},f_{P,\sigma}),(g_{C,x},g_{P,\sigma})\big)&\mapsto
&\sum_{C\subset X, x\in C:\,\pi(x)\in S_\fin}\Res_{C,x}(f_{C,x}g_{C,x}\omega)\\
&&\qquad\ \ +\sum_{P\in X_F}\sum_{\sigma\in S_\infty}
\Res_{P,\sigma}(f_{P,\sigma}g_{P,\sigma}\omega).
\end{matrix}
$$
Here $\Res_{C,x}$, resp. $\Res_{P,\sigma}$, denotes the local pairing on $X$, resp. on 
$X_\sigma$. Moreover, we have

\begin{prop} Let $X$ be an arithmetic surface, $D$ a Weil divisor  and $\omega$  a non-zero 
rational differential on $X$. Then we have

\begin{itemize}
\item [(1)] ([2, Lemma 11])  The pairing $\langle\cdot,\cdot\rangle_\omega$  is 
well defined;

\item [(2)] ([2, Proposition 12])  The pairing $\langle\cdot,\cdot\rangle_\omega$  is non-degenerate;

\item [(3)] ([2, Proposition 15]) With respect to  $\langle\cdot,\cdot\rangle_\omega$, we have
$$\big(\A_{X,01}^{\ar}\big)^\perp\!\!=\!\A_{X,01}^{\ar},\ \big(\A_{X,02}^{\ar}\big)^\perp\!\!=\!\A_{X,02}^{\ar},\  \big(\A_{X,12}^{\ar}(D)\big)^\perp\!\!=\!\A_{X,12}^{\ar}\big((\omega)-D\big).$$
\end{itemize}
\end{prop}

To apply this result, we use the following well-known

\begin{lem} 
With respect to a continue, non-degenerate pairing $\langle\cdot,\cdot\rangle$ on 
a topological space $\mathcal W$, e.g. $\A_{X,012}^\ar$, we have
\begin{itemize}
\item [(1)] If $\,W_1$ and $W_2$  are closed subgroups of $\mathcal W$ such that the spaces
$W_1+W_2$ and $W_1^\perp+ W_2^\perp$ are closed, then 
$$
(W_1+W_2)^\perp=W_1^\perp\cap W_2^\perp\qquad\mathrm{ and}\qquad
(W_1\cap W_2)^\perp=W_1^\perp+ W_2^\perp;
$$
\item [(2)] If  $\,W$ is a closed subgroup of $\mathcal W$, then, algebraically and topologically,
$$
(W^\perp)^\perp=W\qquad\mathrm{ and}\qquad W\simeq \,\widehat {\ \mathcal W\big/ W^\perp\ }.
$$ 
\end{itemize}
\end{lem}

With above preparation, we are ready to prove our theorem.

\begin{proof} 
(A)  Topological duality between $H_\ar^0$ and $H_\ar^2$

By definition of $H^2_{\ar}$ in \S2.1, Lemma 12, and  Proposition 11(3), we have
$$
\begin{aligned}
\widehat {H^2_{\ar}(X,(\omega)-D)}\simeq& \Big(\A_{X,01}^\ar+\A_{X,02}^\ar+\A_{X,12}^\ar((\omega)-D)\Big)^\perp\\
\simeq &\Big(\A_{X,01}^\ar\Big)^\perp\cap\Big(\A_{X,02}^\ar\Big)^\perp\cap\Big(\A_{X,12}^\ar((\omega)-D)\Big)^\perp\\
=&\A_{X,01}^\ar\cap\A_{X,02}^\ar\cap \A_{X,12}^\ar(D)\simeq H^{0}_{\ar}(D).
\end{aligned}
$$
Indeed, by Proposition 8(1),  $\A_{X,01}^\ar+\A_{X,02}^\ar$ is closed, and by Proposition 8(3),  $\A_{X,01}^\ar+\A_{X,02}^\ar+\A_{X,12}^\ar((\omega)-D)$ is closed.
 
\noindent
(B) {Topological duality among $H_\ar^1$}

For $H_\ar^1$, similarly as in [4], see [2, Proposition 16], we have
the following group theoretic isomorphisms:
$$\begin{aligned}
&H^1_{\ar}(X,D)\\
\simeq&\Big(\A_{X, 01}^{\ar}\cap \big(\A_{X, 02}^{\ar}+ \A_{X, 12}^{\ar}(D)\big)\Big)\Big/
\Big(\A_{X, 01}^{\ar}\cap \A_{X, 02}^{\ar}+\A_{X, 01}^{\ar}\cap \A_{X, 12}^{\ar}(D)\Big)\\
\simeq&\Big(\A_{X, 02}^{\ar}\cap \big(\A_{X, 01}^{\ar}+\A_{X, 12}^{\ar}(D)\big)\Big)\Big/\Big(\A_{X, 01}^{\ar}\cap \A_{X, 02}^{\ar}
+\A_{X, 02}^{\ar}\cap \A_{X, 12}^{\ar}(D)\Big).
\end{aligned}$$

Consequently,  we have
$$
\begin{aligned}
\widehat{H^1_{\ar}(X,(\omega)-D)}
=&\Big(\frac{\A_{X, 02}^{\ar}\cap \big(\A_{X, 01}^{\ar}+\A_{X, 12}^{\ar}((\omega)-D)\big)}
{\A_{X, 01}^{\ar}\cap \A_{X, 02}^{\ar}+\A_{X, 02}^{\ar}
\cap \A_{X, 12}^{\ar}((\omega)-D)}\Big)^{\widehat{~~~}}\\
\simeq&\frac{\big(\A_{X, 01}^{\ar}\cap \A_{X, 02}^{\ar}\big)^\perp\cap
\big(\A_{X, 02}^{\ar}\cap \A_{X, 12}^{\ar}((\omega)-D)\big)^\perp}
{\big(\A_{X, 02}^{\ar}\big)^\perp+ \big(\A_{X, 01}^{\ar}+\A_{X, 12}^{\ar}((\omega)-D)\big)^\perp}\\
=&\frac{\big(\A_{X, 01}^{\ar}+ \A_{X, 02}^{\ar}\big)
\cap\big(\A_{X, 02}^{\ar}+ \A_{X, 12}^{\ar}(D)\big)}
{\A_{X, 02}^{\ar}+\A_{X, 01}^{\ar}\cap\A_{X, 12}^{\ar}(D)}\\
\simeq&\frac{\big(\A_{X, 01}^{\ar}+\A_{X, 02}^{\ar}\big)\cap \A_{X, 12}^{\ar}(D)}
{\A_{X, 01}^{\ar}\cap \A_{X, 12}^{\ar}(D)+\A_{X, 02}^{\ar}\cap \A_{X, 12}^{\ar}(D)}
\simeq H^1_{\ar}(X,D).
\end{aligned}
$$ 
Indeed,  the first equality and the last isomorphism are direct consequence of the definition. To 
verify the second isomorphism and the third equality, we use Proposition 8(1) and (2), namely,  
$\A_{X, 01}^{\ar}+ \A_{X, 02}^{\ar}$ and $\A_{X, 02}^{\ar}+ \A_{X, 12}^{\ar}(D)$ are closed, and 
Proposition 8(3), namely, $\A_{X,01}^\ar+\A_{X,02}^\ar+\A_{X,12}^\ar((\omega)-D)$ is closed.   
As for the forth morphisms, both associated quotients spaces are Hausdorff,
since denominators and numerators are all closed and $\A_{X,012}^\ar$ is Hausdorff.
They are also discrete since $H^1_\ar$ is always so. Therefore, this forth morphism is a topological homeomorphism, being a group isomorphism.\end{proof}
\vskip 0.10cm
{\footnotesize{
\noindent
{\bf Acknowledgements.} We would like to thank Osipov for pointing out that the closeness 
condition of $W_1+W_2$ and $W_1^\perp+ W_2^\perp$ 
is necessary for the validity of  Lemma 12.

This work is partially supported by JSPS.
}}

\vskip 1.0cm
\noindent
{\bf REFERENCES}
\vskip 0.10cm
\noindent
[1] A.A. Beilinson, Residues and adeles, Funct. Anal. Pril., 14 (1980), no. 1, 44-45; English transl. in 
Func. Anal. Appl., 14, no. 1 (1980), 34-35.
\vskip 0.20cm
\noindent
[2] K. Sugahara, L. Weng,  Arithmetic Cohomology Groups, preprint 2014, see also arXiv:1507.06074
\vskip 0.10cm
\noindent
[3] D. V. Osipov, A. N. Parshin, Harmonic analysis on local fields and adelic spaces. II.  Izv. Ross. Akad. 
Nauk Ser. Mat. 75 (2011), no. 4, 91--164
\vskip 0.10cm
\noindent
[4] A.N. Parshin,  On the arithmetic of two-dimensional schemes. I. Distributions and residues.  Izv. Akad. 
Nauk SSSR Ser. Mat. 40 (1976), no. 4, 736-773

\vskip 4.80cm
K. Sugahara \& L. Weng

Graduate School of Mathematics,
 
Kyushu University,
 
Fukuoka, 819-0395,
 
Japan
 
E-mails: k-sugahara@math.kyushu-u.ac.jp,
 
\hskip 2.30cm  weng@math.kyushu-u.ac.jp

\end{document}